\documentclass[10pt]{amsart}
\usepackage{epsfig,url}
\usepackage{mathdots}
\usepackage{xcolor}

\newtheorem{theorem}{Theorem}[section]
\newtheorem{lemma}[theorem]{Lemma}
\newtheorem{conjecture}[theorem]{Conjecture}

\newtheorem{corollary}[theorem]{Corollary}
\newtheorem{proposition}[theorem]{Proposition}

\theoremstyle{definition}

\newtheorem{note}[theorem]{Note}

\theoremstyle{remark}

\begin{document}

\title{Arithmetic properties of the sum of divisors}

\author[T. Amdeberhan et al.]{Tewodros Amdeberhan}
\address{Department of Mathematics,
Tulane University, New Orleans, LA 70118}
\email{tamdeber@tulane.edu}

\author[]{Victor H. Moll}
\address{Department of Mathematics,
Tulane University, New Orleans, LA 70118}
\email{vhm@tulane.edu}

\author[]{Vaishavi Sharma}
\address{Department of Mathematics,
Tulane University, New Orleans, LA 70118}
\email{vsharma1@tulane.edu}

\author[]{Diego Villamizar}
\address{Department of Mathematics,
Tulane University, New Orleans, LA 70118}
\email{dvillami@tulane.edu}


\date{\today}


\begin{abstract}
The divisor function $\sigma(n)$ denotes the sum of the divisors of the positive integer $n$. For a prime 
$p$ and $m \in \mathbb{N}$, the $p$-adic valuation of $m$ is the highest power of $p$ which divides $m$. 
Formulas for $\nu_{p}(\sigma(n))$ are established. For $p=2$, these involve only the 
 odd primes dividing $n$.  These 
expressions are used to establish the  bound  $\nu_{2}(\sigma(n)) \leq \lceil\log_{2}(n) \rceil$, with 
equality if and only if $n$ is the product of distinct Mersenne primes, and  for an odd prime $p$, the bound is
$\nu_{p}(\sigma(n)) \leq \lceil \log_{p}(n) \rceil$, with equality related to solutions of  the Ljunggren-Nagell 
diophantine equation.
\end{abstract}

\maketitle

\newcommand{\ba}{\begin{eqnarray}}
\newcommand{\ea}{\end{eqnarray}}
\newcommand{\ift}{\int_{0}^{\infty}}
\newcommand{\nn}{\nonumber}
\newcommand{\no}{\noindent}
\newcommand{\tpri}{\rho}
\newcommand{\lf}{\left\lfloor}
\newcommand{\rf}{\right\rfloor}
\newcommand{\realpart}{\mathop{\rm Re}\nolimits}
\newcommand{\imagpart}{\mathop{\rm Im}\nolimits}

\newcommand{\op}[1]{\ensuremath{\operatorname{#1}}}
\newcommand{\pFq}[5]{\ensuremath{{}_{#1}F_{#2} \left( \genfrac{}{}{0pt}{}{#3}
{#4} \bigg| {#5} \right)}}

\newtheorem{Definition}{\bf Definition}[section]
\newtheorem{Thm}[Definition]{\bf Theorem}
\newtheorem{Example}[Definition]{\bf Example}
\newtheorem{Lem}[Definition]{\bf Lemma}
\newtheorem{Cor}[Definition]{\bf Corollary}
\newtheorem{Prop}[Definition]{\bf Proposition}
\numberwithin{equation}{section}

\section{Introduction}
\label{sec-intro}

For $n \in \mathbb{N}$ and $p$ prime, the $p$-adic valuation of $n$ is the highest power of $p$ 
which divides $n$. It is denoted by $\nu_{p}(n)$.
The goal of the present work is to describe the 
sequence $\left\{ \nu_{p}(\sigma(n)) \right\}$, where $\sigma(n)$ is the sum of divisors of $n$. This is 
a multiplicative function, therefore the $p$-adic valuation of $n$ is given in terms of 
 its  prime factorization 
\begin{equation}
 n = \prod_{j=1}^{s} p_{j}^{\nu_{p_{j}}(n)} 
 \label{prime-fac}
 \end{equation}
 \noindent
 in the form
\begin{equation}
\nu_{p}(\sigma(n)) = \sum_{j=1}^{s} \nu_{p} \left( \sigma \left( p_{j}^{\nu_{p_{j}}(n)}  \right) \right).
\label{sigma-mult}
\end{equation}
\noindent
The work discussed here is part of a general program to examine $p$-adic valuations of classical 
sequences. Examples include  \cite{amdeberhan-2008b}, which deals with Stirling numbers, 
\cite{amdeberhan-2008a} dealing with a family of integers appearing in the evaluation of a rational 
integral and \cite{almodovar-2019a} for valuations of quadratic polynomial sequences. 

\smallskip 

The main results are stated next. 

\begin{theorem}
\label{thm-val1}
Let $n \in \mathbb{N}$ with prime factorization \eqref{prime-fac}. Then 
\begin{equation}
 \nu_2(\sigma(n))=\displaystyle\sum_{\substack{i=1\\ \nu _{p_{i}}(n) \, {{\rm{odd }}} \\p_i \, \rm{ odd}}}^s\nu_2\left(\frac{(\nu_{p_{i}}(n)+1)(p_i+1)}{2}\right).
 \label{val2-sigma}
 \end{equation}
\end{theorem}

The result for an odd prime is given in terms of the order of $q$ modulo $p$, denoted by 
$r = Ord_{p}(q)$. This is the minimal positive integer $r$ such that $q^{r} \equiv 1 \bmod p$.

\begin{theorem}
\label{thm-val2}
Assume $p, \, q$ are primes with $p$ odd. Let  $r = Ord_{p}(q)$. Then 
\begin{equation*}
    \nu _p\left (\sigma (q^k)\right ) =\begin{cases}
    \nu _p (k+1), & \text{if }q\equiv 1 \pmod p\\
    0, & \text{if} \,\, p=q \text{ or }(q\not \equiv 1 \pmod p \text{ and }k\not \equiv -1 \bmod r)\\
\nu _p\left (k+1\right )+\nu _p(q^r-1), & \text{otherwise.}
    \end{cases}
\end{equation*}
\noindent
The analog of \eqref{val2-sigma} is then  obtained from \eqref{sigma-mult} to produce 
\begin{multline}
    \nu _p(\sigma (n))=\sum _{\substack{q|n\\q\equiv 1\bmod p}} \nu _p\left (\nu _q(n)+1\right ) \\
    +\sum _{\substack{q|n\\ p\neq q\\ q\not \equiv 1\bmod p \\ Ord_p(q)\, |\, (\nu _q(n)+1)}}\left (\nu _p\left (\nu _q(n)+1\right )+\nu _p\left (q^{Ord_p(q)}-1\right ) \right ).
\end{multline}
\end{theorem}

The expressions above will be used to provide bounds for the valuations of $\sigma(n)$. 

\begin{theorem}
\label{thm-bound1}
For $n \in \mathbb{N}$, 
\begin{equation}
\nu_{2}( \sigma(n)) \leq \lceil \log_{2}n \rceil,
\end{equation}
\noindent
and  equality holds if and only if $n$ is the product of distinct Mersenne primes. 
\end{theorem}

The bound for the 
$p$-adic valuation of $\sigma(n)$, for a fixed odd prime $p$, 
is given in Theorem \ref{thm-bound2}. The proof presented here is valid under the following three 
assumptions on $n$: 
\begin{enumerate}
\item every prime $q$ dividing $n$ satisfies $\nu_{p}(\sigma(q^{\nu_{q}(n)})) \leq 
\lfloor \log_{p}(q^{\nu_{q}(n)}) \rfloor $ or 
    \item every prime $q$ dividing $n$ which satisfies  
    $\nu_{p}(\sigma(q^{\nu_{q}(n)})) = 
\lceil \log_{p}(q^{\nu_{q}(n)}) \rceil $
  is  less than $p$, or
   \item there is a unique prime $q>p$ dividing $n$ such that  
     $\nu_{p}(\sigma(q^{\nu_{q}(n)})) = 
\lceil \log_{p}(q^{\nu_{q}(n)}) \rceil$.
\end{enumerate}

\begin{theorem}
\label{thm-bound2} 
Let $p>2$ be prime. Assume $n$ satisfy one of the conditions $(1), \, (2)$, or $(3)$ given above. Then 
\begin{equation}
\label{bd1}
\nu_{p}( \sigma(n)) \leq \lceil \log_{p}n \rceil.
\end{equation}
\noindent
Equality occurs if  every prime factor $q$ of $n$ is solution of the Ljunggren-Nagell equation 
$\begin{displaystyle} \frac{q^{k+1}-1}{q-1} = p^{s}\end{displaystyle}$, for some $k, \, s \in \mathbb{N}.$ 
\end{theorem}

\begin{note}
Examples of  type $(1)$ are easy to produce. Take $p=5$ and $n = 2^{3}$. Then 
$\sigma(2^3) = 15$ so that  $\nu_{5}(\sigma(2^3)) = 1$ and since  $ \log_{5}(2^3) \sim 1.89$
the value $\lfloor \log_{5}(2^3) \rfloor= 1$ shows equality is achieved.  For type $(2)$ take $n = 173200 = 
2^4 \cdot 5^2 \cdot 433$ and $p=31$. Then $\nu_{31}(\sigma(2^4)) = \lceil \log_{31}2^4 \rceil = 1, 
\nu_{31}(\sigma(5^2)) = \lceil \log_{31}5^2 \rceil = 1$ and  
$\nu_{31}(\sigma(433)) = 1 <  \lceil \log_{31}2^4 \rceil = 2$. The bound in Theorem \ref{thm-bound2} below 
holds for this value of $n$, 
since $\nu_{31}(\sigma(173200)) = 3  \leq \lceil \log_{31}173200 \rceil = 4$. There are no 
examples of type (3) up to $n \leq 5 \times 10^6$.
\end{note}

\begin{conjecture}
\label{conj1}
Let $p>2$ be a fixed prime. Then every $n \in \mathbb{N}$ satisfies condition $(1)$ or $(2)$. 
Therefore, the bound \eqref{bd1} holds for every $n \in \mathbb{N}.$
\end{conjecture}

A complete characterization of indices $n$ where equality is achieved in \eqref{bd1} remains an open 
question.  Partial results are presented in Lemma \ref{lemma-bd1}. 

\section{Auxiliary results}
\label{sec-auxiliary}

This section contains some elementary results on the polynomials $[x]_{n}$ defined by 
\begin{equation}
[x]_{n} = \sum_{k=0}^{n} x^{k} = \frac{x^{n+1}-1}{x-1}.
\label{def-poly1}
\end{equation}
\noindent
This is usually known as $q$-bracket and denoted by $[n+1]_x,$ where $x=q. \, $ Information 
about them appears in \cite{andrews-1999a} and \cite{gasper-1990a}.
The polynomials $[x]_{n}$ play a crucial role in the proof of Theorem \ref{thm-val1}. 

The main connection to the current problem is the observation that, for $q$ prime and arbitrary 
$k \in \mathbb{N}$, one has 
\begin{equation}
\sigma(q^{k}) = [q]_{k}.
\end{equation}

\smallskip 

The proofs of the first three results are elementary. 

\begin{lemma}
\label{lemma0}
Assume $x, \, \ell, \, k \in \mathbb{N}$ with $x$ odd, $\ell$ even. Then 
$[ x^{k} ]_{\ell}$ is odd.
\end{lemma}

\begin{lemma}
\label{lemma1}
For $n$ even
\begin{equation}
[x]_{n} [-x]_{n} = \left[ x^{2} \right]_{n}.
\end{equation}
\end{lemma}

\begin{lemma}
\label{lemma2}
For $n \in \mathbb{N}$
\begin{equation}
\left[ x^{2} \right]_{n} [x]_{1} = [x]_{2n+1}.
\end{equation}
\end{lemma}

\begin{lemma}
\label{lemma3}
For $r \in \mathbb{N}$, 
\begin{equation}
[x]_{2^{r}-1} = \prod_{i=0}^{r-1} \left( x^{2^{i}} + 1 \right).
\end{equation}
\end{lemma}
\begin{proof}
This follows directly from 
\begin{eqnarray}
[x]_{2^{r-1}} & = & \frac{x^{2^{r}}-1}{x-1} = 
\frac{x^{2^{r}}-1}{x^{2^{r-1}}-1} \times 
\frac{x^{2^{r-1}}-1}{x^{2^{r-2}}-1} \times \cdots \times \frac{x^{4}-1}{x^{2}-1} \frac{x^{2}-1}{x-1}.
\end{eqnarray}
\end{proof}

\noindent
This identity in Lemma \ref{lemma3} reflects the 
fact that  that every number $0\leq m< 2^r$ has a unique expansion in base $2$.

\begin{lemma}
\label{lemma4}
Assume $k \in \mathbb{N}$ with $k+1=2^a\cdot b,$ with $b$ odd  so that $\nu_{2}(k+1) = a$. 
Then 
\begin{equation}
[x]_{2k+1} = \left[ x^{2^{a+1}} \right]_{b-1} \times 
 \prod_{i=0}^{a} \left( x^{2^{i}}+1 \right).
\end{equation}
\end{lemma}
\begin{proof}
Every number $j$ in the range $0 \leq j \leq 2k+1$ has a 
unique representation  in the form 
$j = 2^{a+1}s + r$, with $0 \leq r \leq 2^{a+1} - 1$ and $0 \leq s< b$.  Therefore 
\begin{eqnarray*}
[x]_{2k+1} & = & \sum_{j=0}^{2k+1} x^{j} = \sum_{j=0}^{2k+1} x^{2^{a+1}s+r} \\
& = & \left( \sum_{r=0}^{2^{a+1}-1} x^{r} \right) \times \left( \sum_{s=0}^{b-1} 
\left( x^{2^{a+1}} \right)^{s} \right) \\
& = & [x]_{2^{a+1}-1} \times  \left[ x^{2^{a+1}} \right]_{b-1} 
\end{eqnarray*}
\noindent 
and the result follows by using Lemma \ref{lemma3} on the first factor. 
\end{proof}

\section{The $2$-adic valuation of $\sigma(n)$}
\label{sec-2adic}

This section presents a proof of Theorem \ref{thm-val1}.  In this section, $\tpri$ denotes an odd prime. 

\begin{lemma}
The value $\nu_{2}(\sigma(n))$ depends only on the odd part of $n$; that is, if $n = 2^{a}b$, with $b$ odd, then 
$\nu_{2}(\sigma(n)) = \nu_{2}(\sigma(b))$.
\end{lemma}
\begin{proof}
The result follows from the multiplicativity of $\sigma$ and 
\begin{equation}
\sigma \left( 2^{a} \right) = 1 + 2 + 2^{2} + \cdots + 2^{a} \equiv 1 \bmod 2.
\end{equation}
\end{proof}

\begin{theorem}
\label{thm-form1}
Assume $\tpri$ is an odd prime and $\ell \in \mathbb{N}$. Then 
\begin{equation}
\nu_{2} \left(\sigma\left(\tpri^{\ell}\right)\right)  = \begin{cases}
0 & \text{ if } \ell \text{ is even}, \\
\nu_{2}(\ell+1) + \nu_{2}(\tpri+1) -1 & \text{ if } \ell \text{ is odd}.
\end{cases}
\end{equation}
\end{theorem}
\begin{proof}
For $\ell$  even, the result  follows from 
$\sigma(\tpri^{\ell}) = 1 + \tpri  + \cdots + \tpri^{\ell} \equiv 1 \bmod 2$. Now assume $\ell$ is odd. If 
$\ell = 4r+1$, then 
\begin{eqnarray}
\nu_{2} \left( \sigma(\tpri^{4r+1}) \right) & = & \nu_{2} ( [\tpri]_{4r+1}) \,\, \text{ since } \tpri \text{ is prime} \\
& = & \nu_{2} \left( [\tpri^{2}]_{2r} [\tpri]_{1} \right) \,\, \text{ by Lemma }\ref{lemma2}  \nonumber \\
& = & \nu_{2} \left( [\tpri]_{2r} \, [-\tpri]_{2r} \,\, [\tpri]_{1} \right) \,\, \text{ by Lemma } \ref{lemma1}  \nonumber \\
& = & \nu_{2}([\tpri]_{2r} ) + \nu_{2}([-\tpri]_{2r}) + \nu_{2}([\tpri]_{1}). \nonumber 
\end{eqnarray}
\noindent
Since $\tpri$ is an odd prime, $[\tpri]_{2r} \equiv [-\tpri]_{2r} \equiv 1 \bmod 2$, therefore 
\begin{equation}
\nu_{2}(\sigma(\tpri^{\ell})) = \nu_{2}([\tpri]_{1}) = \nu_{2}(\tpri+1).
\end{equation}
\noindent
Now $\ell + 1 \equiv 2 \bmod 4$ implies $\nu_{2}(\ell+1) = 1$ and gives the result when 
$\ell \equiv 1 \bmod 4$. 

In the case $\ell \equiv 3 \bmod 4$, write $\ell = 2k+1$ and 
$k = 2^{a}b-1$, with $b$ odd (and so $a = \nu_{2}(k+1)=\nu _2(\ell +1)-1$). Lemma 
\ref{lemma4} now gives (with $k = 2r+1$ so that $2k+1 = 4r+3 = \ell$),
\begin{equation}
[\tpri]_{\ell} = [\tpri]_{2k+1} = [\tpri]_{2^{a+1}-1} \times [\tpri^{2^{a+1}} ]_{b-1}.
\end{equation}
\noindent
Lemma \ref{lemma0} shows that the second factor is odd, since $b-1$ is even. It follows that 
\begin{equation}
\nu_{2}([\tpri]_{\ell}) = \nu_{2}\left( [\tpri]_{2^{a+1}-1} \right).
\end{equation}
\noindent
The identity 
\begin{equation}
[\tpri]_{2^{a+1}-1} =  \frac{\tpri^{2^{a+1}}-1}{\tpri-1} = (\tpri+1) \times \prod_{j=1}^{a} \left( \tpri^{2^{j}}+1 \right)
\end{equation}

\noindent
in Lemma \ref{lemma3} yields 
\begin{equation}
\label{last-sum}
\nu_{2}([\tpri]_{\ell}) = \nu_{2}(\tpri+1) + \sum_{j=1}^{a} \nu_{2} \left( \tpri^{2^{j}}+1 \right).
\end{equation}
\noindent
For the last term, write $\tpri = 4u+v$ with $v=1$ or $3$. Then, for $j \geq 1$, 
\begin{equation}
\tpri^{2^{j}} \equiv v^{2^{j}} = (v^{2})^{2^{j-1}} \equiv 1 \bmod 4,
\end{equation}
\noindent
so that 
$\begin{displaystyle} \nu_{2} ( \tpri^{2^{j}}+1  ) = 1 \end{displaystyle}$ and \eqref{last-sum} becomes 
$\nu_{2}([\tpri]_{\ell}) = \nu_{2}(\tpri+1) + a$. Finally observe that 
$a = \nu_{2}(k+1)$ and $k+1 = \tfrac{1}{2}(\ell+1)$. The proof is complete.
\end{proof}

Theorem \ref{thm-form1} is now used to give a proof of  an elementary result.

\begin{corollary}
The value $\sigma(n)$ is odd if and only if $n = 2^{a}m^{2}$, with $m$ odd.
\end{corollary}
\begin{proof}
Write $n = 2^{a} p_{1}^{\alpha_{1}} \cdots p_{r}^{\alpha_{r}}$ and observe that 
\begin{equation}
\nu_{2}(\sigma(n))  = \nu_{2}(\sigma(2^{a}) )+  \sum_{j=1}^{r} \nu_{2}(\sigma \left( p_{j}^{\alpha_{j}} \right)).
\end{equation}
The  first term vanishes since $\sigma(2^{a}) = 2^{a+1}-1$ is odd. Theorem \ref{thm-form1} shows that the 
sum can be restricted to those primes with $\alpha_{j}$ is odd.  Then 
\begin{equation}
\label{sum-3.11}
\nu_{2}(\sigma(n)) = 
 \sum_{\substack{j=1 \\ \alpha_{j} odd}}^{r}  \left[ \nu_{2}(\alpha_{j}+1) + \nu_{2}(p_{j}+1) - 1 \right] 
\end{equation}
\noindent
Since $\alpha_{j}$ and $p_{j}$ are odd, it follows that $\nu_{2}(\alpha_{j}+1) > 0$ 
and $\nu_{2}(p_{j}+1) > 0$.  Therefore it follows that the sum in \eqref{sum-3.11} must be empty and there 
are no $\alpha_{j}$ odd in the factorization of $n$. This completes the proof.
\end{proof}

\section{Bounds on $2$-adic valuations}
\label{sec-bounds}

This section presents a proof of Theorem \ref{thm-bound1}. It states that 
$\nu_{2}(\sigma(n)) \leq \lceil \log_{2}n \rceil$ and determines conditions for equality to hold. 

\begin{proof}
The proof is divided into a sequence of steps.  

\smallskip 

\noindent
\texttt{Step 1}. Suppose $n = 2^{t}$, with $t \geq 0$. Then $\sigma(n) = 1 + 2 + \cdots + 2^{t} \equiv 1 \bmod 2$
and 
\begin{equation}
0 = \nu_{2}(\sigma(n)) \leq t  = \lfloor \log_{2}n \rfloor = \lceil \log_{2}n \rceil.
\end{equation}

\smallskip 

\noindent
\texttt{Step 2}: Suppose $n=p$ is an odd prime. Choose 
$\alpha$ such that $2^{\alpha} < p < 2^{\alpha+1}$. Since  $p+1 \leq 2^{\alpha+1}$, 
\begin{equation}
\nu_{2}(\sigma(p)) = \nu_{2}(p+1)  \leq \alpha+1 = \lceil \log_{2}p \rceil
\end{equation}
\noindent
and the  inequality follows. In addition, if $p \neq 2^{\alpha+1} - 1$, then 
$p+1 \leq 2^{\alpha+1} - 1$ and this implies $\nu_{2}(p+1) \leq \alpha \leq \lfloor \log_{2} p \rfloor$; 
where the last step follows  from  $2^{\alpha} < p$. This gives a stronger inequality: 
$\nu_{2}(\sigma(n)) \leq \lfloor \log_{2} n \rfloor$.

On the other hand, if $p = 2^{\alpha +1}-1$, that is $p$ is a Mersenne prime, 
\begin{equation}
\nu_{2}(\sigma(p)) = \nu_{2}(p+1) = \alpha + 1 = \lceil \log_{2}n \rceil.
\end{equation}
\noindent
Conversely, if $\nu_{2}(p+1) = \alpha+1$, then $p+1 = 2^{\alpha+1} b$, with $b$ odd. The bounds 
on $p$ give $b=1$, so that $p = 2^{\alpha+1}-1$ and $p$ is a Mersenne prime. This proves the 
result when $n$ is a prime.

\smallskip 

\noindent
\texttt{Step 3}. The inequality $\nu_{2}( \sigma(n) ) \leq \lfloor \log_{2} n \rfloor$ holds 
if $n = p^{j}$, with $j$ even. This is elementary: 
$\sigma(n) = 1 + p + p^{2} + \cdots + p^{j} \equiv j+1 \equiv 1 \bmod 2$ and thus 
$\nu_{2}(\sigma(n)) = 0$. 

\smallskip 

\noindent
\texttt{Step 4}. The next case is $n=p^{3}$, with $p$ an odd prime. Assume $p$ is not a 
Mersenne prime and 
use Theorem \ref{thm-form1}  to obtain 
\begin{eqnarray}
\nu_{2}(\sigma (n)) = \nu_{2}(\sigma(p^{3})) &  = & 1  + \nu_{2}(p+1)  =  1 + \nu_{2}(\sigma(p)) \nonumber \\
& \leq & 1 + \lceil \log_{2}p \rceil \quad \text{by Step 2} \nonumber \\
& < &  \lceil \log_{2}(p^{3}) \rceil = \lceil \log_{2}n \rceil. \nonumber
\end{eqnarray}
\noindent
This gives $\nu_{2}(\sigma (n)) \leq  \lfloor  \log_{2}n \rfloor$. On the other hand, if $n = p^{3}$ with 
$p = 2^{t}-1$ a Mersenne prime 
\begin{equation}
\nu_{2}(\sigma(n)) = 1 + \nu_{2}(p+1) = 1 + t  = 1 + \lceil \log_{2} p \rceil < \lceil \log_{2} n 
\rceil,
\end{equation}
\noindent
and the inequality is strict again. This proves the result when $n = p^{3}$.

\smallskip 

\noindent
\texttt{Step 5}. Now assume $n = p^{j}$ with $j \geq 5$ odd, say $j = 2k+1$. Then  
\begin{eqnarray*}
\lceil \log_{2}(p^{2k+1}) \rceil  & \geq & \log_{2}(p^{2k+1}) = k \log_{2}p + (k+1) \log_{2}p \\
& \geq &  k \log_{2}p + k+1 \geq k \log_{2}p + \nu_{2}(k+1).
\end{eqnarray*}
\noindent 
For $k \geq 2, \,\, k \log_{2}p \geq \log_{2}(p^{2}) > \log_{2}(p+1)$ and
$k \log_{2} p > \log_{2}(p+1)$. It follows that
\begin{equation*}
\lceil \log_{2}(p^{2k+1}) \rceil  > \nu_{2}(p+1) + \nu_{2}(k+1).
\end{equation*}
\noindent
Theorem \ref{thm-form1} shows that $\nu_{2}(p+1) + \nu_{2}(k+1) = \nu_{2}(\sigma (p^{2k+1}))$ and 
the result also holds in this case. 
%
%
 
 \smallskip 
 
 \noindent
 \texttt{Step 6}. Proceed by induction on $n$. The assumption is that, 
  in the prime factorization of $n = p_{1}^{a_{1}} p_{2}^{a_{2}} \cdots 
 p_{r}^{a_{r}}$, the exponent corresponding to a Mersenne prime is at least $2$. Then 
 $\nu_{2}(\sigma(n)) \leq \lfloor \log_{2} n \rfloor$. Write  $n = m \times p^{a}$, where 
 $\gcd(m,p) = 1$. Then $\nu_{2}(\sigma(m)) \leq \lfloor \log_{2}m \rfloor$ (by induction) and 
 $\nu_{2}(\sigma(p^{a})) \leq \lfloor \log_{2}(p^{a}) \rfloor$ by Steps 2-4. Then
 \begin{equation}
 \nu_{2}(\sigma(n)) = \nu_{2}(\sigma(m)) + \nu_{2}(p^{a}) \leq 
 \lfloor \log_{2}m \rfloor + \lfloor \log_{2}(p^{a}) \rfloor \leq \lfloor \log_{2} n \rfloor 
 \end{equation}
 \noindent
 since $\lfloor x \rfloor + \lfloor y \rfloor \leq \lfloor x+ y \rfloor$.
 
  \smallskip 
 
 \noindent
 \texttt{Step 7}. Assume now that  $\begin{displaystyle} n = \prod_{i=1}^{r} p_{i} \end{displaystyle}$, where 
 $p_{i} = 2^{q_{i}}-1$ are distinct Mersenne primes. This implies  $q_{i}$ must be a prime number. Now 
 \begin{eqnarray*}
 \nu_{2}(\sigma(n)) & = & \sum_{i=1}^{r} \nu_{2}(2^{q_{i}}) \\
 & = & \sum_{i=1}^{r} \log_{2}(p_{i}+1) = \sum_{i=1}^{r} \log_{2}(p_{i}) + 
 \left( \log_{2}(p_{i}+1) - \log_{2}p_{i} \right) \\
 & = & \log_{2}n  + \sum_{i=1}^{r} \log_{2} \left( 1 + \frac{1}{p_{i}} \right) \\
 & < & \log_{2}n + \frac{1}{\log 2} \sum_{i=1}^{r} \frac{1}{p_{i}} \\
 & < & \log_{2}n +1,
 \end{eqnarray*}
 \noindent
 using the bound 
 \begin{equation}
 \sum_{i=1}^{r} \frac{1}{p_{i}} < 0.5165 < \log 2 
 \end{equation}
 \noindent
 obtained by Tanaka \cite{tanaka-2017a}. This implies 
 \begin{equation}
\log_{2}n <  \nu_{2}(\sigma(n)) < \log_{2}n + 1
\end{equation}
\noindent
showing that $\nu_{2}(\sigma(n)) = \lceil \log_{2} n \rceil$, since $n$ is not a power of $2$. 

 \smallskip 
 
 \noindent
 \texttt{Step 8}. Finally, factor $n = n_{1}n_{2}$, where $n_{2}$ is the product of all Mersenne 
 primes $p$ dividing $n$ to the first power; that is, $\nu_{p}(n) = 1$. The factor $n_{1}$ contains all other 
 primes, so that $\gcd(n_{1},n_{2}) = 1$. Then 
 \begin{eqnarray*}
 \nu_{2}(\sigma(n)) & = & \nu_{2}(\sigma(n_{1})) + \nu_{2}(\sigma(n_{2})) \\
 & \leq & \lfloor \log_{2} n_{1} \rfloor + \lceil  \log_{2}n_{2} \rceil \\
 & = & \left\lfloor \lfloor \log_{2} n_{1} \rfloor  + \lceil \log_{2}n_{2} \rceil \right \rfloor \\
 & \leq & \lfloor \log_{2}n_{1} + \log_{2}n_{2} + 1 \rfloor \\
 & = & \lfloor \log_{2}n + 1 \rfloor \\
 &= &  \lfloor \log_{2} n \rfloor +1.
 \end{eqnarray*}
 \noindent
 Since $\log_{2}n \not \in \mathbb{Z}$ it follows that 
 \begin{equation}
 \lfloor \log_{2}n \rfloor < \log_{2}n < \lceil \log_{2} n \rceil 
 \end{equation}
 \noindent
 and thus $\lfloor \log_{2}n \rfloor +1 = \lceil \log_{2}n \rceil$. This  completes the proof.
 \end{proof}

\section{The formula for an odd prime}
\label{sec-oddprime}

This section presents the proof of Theorem \ref{thm-val2}. The discussion begins with
preliminary results, which admit elementary proofs. 

\begin{lemma}
\label{lemma5}
For $r \in \mathbb{N}$ and $p>1$, 
\begin{equation}
[x]_{p^{r}-1} = \prod_{i=0}^{r-1} \left[ x^{p^{i}} \right]_{p-1}.
\end{equation}
\end{lemma}
\begin{lemma}
\label{lemma6}
Assume $k \in \mathbb{N}$ has $\nu_{p}(k+1) = a$; that is, $k = p^{a}b - 1$ with $(b,p)=1$. Then 
\begin{equation}
[x]_{k} = \left[ x^{p^{a}} \right]_{b-1} \times 
 \left [x\right ]_{p^{a}-1}.
\end{equation}
\end{lemma}

\begin{proposition}
\label{lemma-pad1}
Let $p$ be an odd prime and let $q\neq p$ be prime. For $n\geq 0$,
\begin{equation}
    \nu _{p}\left ([q^{p^n}]_{p-1}\right )=\begin{cases}
    1, & \text{if }q\equiv 1\pmod p\\
    0, & \text{otherwise}.
    \end{cases}
\end{equation}
\end{proposition}
\begin{proof}
Write  $q=q_1+q_2\cdot p+q_3\cdot p^2$ with $0\leq q_2<p$ and $0<q_1<p.$ 
Assume first  $q_1\neq 1,$ then
\begin{align*}
    \left [q^{p^n}\right ]_{p-1} & =(q-1)^{-1}(q^{p\cdot p^n}-1)\\
    &\equiv (q_1-1)^{-1}(q_1^{p^{n+1}}-1) \pmod p\\
    &\equiv (q_1-1)^{-1}(q_1-1) \pmod p\\
    &\equiv 1 \pmod p,
\end{align*}
and so  $\nu _p(\left [q^{p^n}\right ]_{p-1})=0$. Now, in the case $q \equiv 1 \bmod p$, then 
$q_{1} = 1$ and 
\begin{equation*}
    \left [q^{p^n}\right ]_{p-1}=\sum _{i=0}^{p-1}q^{i\cdot p^n} 
    \equiv \sum _{i=0}^{p-1}1^{i\cdot p^n}
    \equiv 0 \pmod p,
\end{equation*}
and therefore $\nu _p(\left [q^{p^n}\right ]_{p-1})\geq 1$. 
Moreover, 
\begin{align*}
    \left [q^{p^n}\right ]_{p-1}=\sum _{i=0}^{p-1}q^{i\cdot p^n}
    &\equiv \sum _{i=0}^{p-1}(1+q_2\cdot p)^{i\cdot p^n} \pmod {p^2}\\
    &\equiv 1+\sum _{i=1}^{p-1}\sum _{j=0}^{i\cdot p^n}\binom{i\cdot p^n }{j}(q_2\cdot p)^j \pmod {p^2}\\
    &\equiv 1+\sum _{i=1}^{p-1}\left (1+i\cdot p^n\cdot q_2\cdot p \right ) \pmod {p^2}\\
    &\equiv p+(p^{n+2}q_2(p-1)-1)2^{-1} \equiv p-2^{-1}\not \equiv 0 \pmod{p^2}.
\end{align*}
This completes the proof.
\end{proof}

The results above are now used to complete the proof of Theorem \ref{thm-val2}. This 
provides an expression for $\nu_{p}(\sigma(n))$, for $p$ an odd prime.  By the multiplicative 
property of the $\sigma$-function, it suffices to consider the case $n = q^{k}$. The proof is 
divided into a sequence of steps. 

\smallskip 

\noindent
\texttt{Step 1}. Assume $p=q,$ then $\sigma (q^k)\equiv 1\pmod p$ and the valuation is $0.$

\smallskip

\noindent
\texttt{Step 2}. Assume  $q\equiv 1 \pmod p$ and define $k=p^{a}\ell -1$ for $(\ell,p)=1$, so that 
 $a=\nu _p(k+1)$. Then $\nu_{p}(\sigma(q^{k})) = \nu_{p}(k+1)$.
 \begin{proof}
 Start with
\begin{align*}
    \nu _p(\sigma (q^k))&=\nu _p([q]_k)
    =\nu _p([q]_{p^{a}\ell -1})\\
    &=\nu _p([q^{p^a}]_{\ell-1}[q]_{p^{a}-1}) & \text{By lemma }\ref{lemma6}\\
    &=\nu _p([q^{p^a}]_{\ell-1})+\nu _p([q]_{p^{a}-1})\\
    &=\nu _p([q]_{p^a-1})  &\text{ Because }[q^{p^a}]_{\ell-1}\equiv \ell \pmod p\\
    &=\sum _{i=0}^{a-1}\nu _p\left (\left [q^{p^i}\right ]_{p-1}\right )& \text{By lemma }\ref{lemma5}\\
    &=a, & 
\end{align*}
\noindent
using $q \equiv 1 \bmod p$ and  Proposition \ref{lemma-pad1} in the last step.
Since $a = \nu_{p}(k+1)$, the argument is complete. 
\end{proof}

\smallskip

\noindent
\texttt{Step 3.} Assume $q\not \equiv 1\pmod p$ and $r \not |  \,\,  k+1$. Then 
$\nu_{p}(\sigma(q^{k})) = 0$. 
\begin{proof}
Write 
$k=p^a\ell-1$, with $a=\nu _p(k+1)$, to  conclude that $r\not | \,\, \ell$ and then 
\begin{equation}
[q^{p^a}]_{\ell-1}\equiv (q-1)^{-1}(q^{\ell}-1)\not \equiv  0 \pmod p.     
\end{equation}  
Proposition \ref{lemma-pad1} now implies $\nu _p(\sigma (q^k))=0$.
\end{proof}

\smallskip

\noindent
\texttt{Step  4.} Let $q\not \equiv 1 \pmod p$ and $r|(k+1)$.  Then 
$\nu _p(\sigma (q^k))=\nu _p(k+1)+\nu _p(q^r-1).$

\begin{proof}
Write  $k+1=p^a\ell =p^ar\ell _1$ 
for some $\ell _1$. The  previous computations gives 

\begin{align*}
    \nu _p(\sigma (q^k)) & = \nu _p([q]_k)\\
    &= \nu _p([q]_{p^a\ell -1})\\
    &= \nu _p([q]_{p^a-1}) + \nu _p([q^{p^a}]_{\ell -1})\\
    &=\nu _p([q^{p^a}]_{\ell -1}).
\end{align*}
Then 
\begin{align*}
\nu _p \left( \frac{\sigma (q^k)}{q^r-1} \right) 
&=\nu _p \left( \frac{[q^{p^a}]_{\ell -1}}{q^r-1} \right) \\
&=\nu _p \left( \frac{(q^{r})^{\ell _1 p^a}-1}{(q^{p^a}-1)(q^r-1)} \right)   \\
&=\nu _p \left( \frac{(q^{r})^{\ell _1 p^a}-1}{q^r-1} \right)  & \text{ because }q\not \equiv 1 \pmod p\\
&=\nu _p([q^r]_{\ell _1p^a-1})\\
&=\nu _p([q^r]_{p^a-1}[q^{rp^a}]_{\ell _1-1}) & \text{By lemma }\ref{lemma6}\\
&=\nu _p([q^r]_{p^a-1})\\
&=\sum _{i=0}^{a-1}\nu _p([q^{rp^i}]_{p-1}) & \text{ By lemma }\ref{lemma5}\\
&=a,
\end{align*}
using Proposition \ref{lemma-pad1} in the last step since $q^{r} \equiv 1 \bmod p$. Then \newline 
$\begin{displaystyle}
\nu _p(\sigma (q^k))=a+\nu _p(q^r-1)=\nu _p(k+1)+\nu _p(q^r-1),
\end{displaystyle}$
as claimed.
\end{proof}

\smallskip 

The proof of Theorem \ref{thm-val2} is now complete.

\section{Bounds on $p$-adic valuations}
\label{sec-p-bounds}

The goal of this section is to establish the  bound 
\begin{equation}
\label{odd-bound0}
\nu_{p}(\sigma(n)) \leq  \lceil \log_{p}n \rceil,
\end{equation}
\noindent
given in Theorem \ref{thm-main}. The proof presented here contains the restrictions on $n$ 
described in Conjecture \ref{conj1}.  These 
restrictions, assumed to hold for all $n$, come from the technique used in the proof.

\begin{lemma}
\label{lemmaequ}
Let $p, \, q, \, k \in \mathbb{N}$ with $p \neq q$ prime. The following statements are equivalent:

\noindent
$\begin{displaystyle} (1) \,\, \nu _p(\sigma(q^k))\geq \lceil \log _p(q^k) \rceil \end{displaystyle}$, 

\noindent
$\begin{displaystyle} (2) \,\, \sigma (q^k)= p^s\end{displaystyle}$ for some $s\geq 1.$ 
\end{lemma}
\begin{proof}
Assume $(1)$ holds. Since  $q\geq 2,$ it follows that  $$\frac{\sigma (q^k)}{q^k}=\frac{q^{k+1}-1}{q^{k+1}-q^k}<2,$$
so that  $\sigma(q^k)<2q^k.$ Now  write  $\sigma(q^k)=p^s x,$ with $(x,p)=1.$ Then 
$\begin{displaystyle} \sigma(q^k)=p^s x<2 q^k.\end{displaystyle}$
Assume that $x\geq 2$, then $p^s x<2 q^k\leq x q^k,$ and then
$\begin{displaystyle} s=\log _p(p^s)<\log_p(q^k). \end{displaystyle}$ From 
 $s=\nu _p(\sigma(q^k))$ it follows that  $\nu _p(\sigma(q^k))<\log_p(q^k)$, for $x>1$. Thus, if 
 $(1)$ holds, it  must be that $x=1$; that is, $\sigma(q^{k}) = p^{s}$.
 
 Conversely, if $\sigma(q^{k}) = p^{s}$, then 
 $\log_p(q^k)<s=\nu _p(\sigma(q^k))$ and $\lceil \log _p(q^k)\rceil\leq \nu _p(\sigma(q^k))$ follows 
 from here.
\end{proof}


The proof of Theorem \ref{thm-bound2} under the conditions stated in Conjecture \ref{conj1} is presented 
next. As before, $p$ denotes an odd prime and 
$q \neq p$ is a second prime and $r = \text{Ord}_{p}(q)$. 

\smallskip

Some preliminary results, required for the proof of Theorem \ref{thm-main}, are established next.  The 
discussion is divided into two types. 

\smallskip

\noindent
\texttt{Type 1}. Assume $k \in \mathbb{Z}^{+}$ is such that \textit{one of the following conditions hold}: 
\begin{enumerate}
\item{either $p=q$ or $q \not \equiv 1 \bmod p$ and $r$ does not divide $k+1$,}
\item{$q \equiv 1 \bmod p$,}
\item{$k=1$ and $q>2$,}
\item{$q \not \equiv 1 \bmod p, \,\, r$ divides $k+1$ and $r \neq k+1$.}
\end{enumerate}

\noindent
\texttt{Type 2}.  Assume $k \in \mathbb{Z}^{+}$ is such that \textit{the following two conditions hold}:
\begin{enumerate}
\item{$r = k+1$,}
\item{$q \not \equiv 1 \bmod p$.}
\end{enumerate}

The first main result of this section is 

\begin{theorem}
\label{thm-types}
Assume $p, \, q, \, k$ are as before. Then,
\begin{enumerate}
\item{In type 1, the inequality $\nu_{p}(\sigma(q^{k})) \leq \lfloor \log_{p} q^{k} \rfloor$ holds,}
\item{In type 2, the inequality $\nu_{p}(\sigma(q^{k})) \leq \lceil \log_{p} q^{k} \rceil$ holds.}
\end{enumerate}
\end{theorem}

\begin{lemma}
Assume $p, \, q, \, r$  are as above and satisfy  the equivalent statements of Lemma 
 \ref{lemmaequ} hold. Assume 
  $q\not \equiv 1\pmod p$ and $\nu _p(k+1)=0$. Then $k+1=r$.
\end{lemma}
\begin{proof}
If $k+1$ is not divisible by $r$, then 
$\nu_p(\sigma(q^k))=0$ and part (1) of Lemma \ref{lemmaequ} imply that 
 $\lceil \log _p(q^k)\rceil =0$. This yields $\log _p(q^k)=0$ hence $k=0$.  Since this value is excluded, it 
 follows that $k+1$ must be divisible by  $r$. 

 Therefore define 
$m$ by $k+1=r\cdot m$.  It will be shown
that $m=1$. Start with 
\begin{equation}
   \nu_p(\sigma(q^k))=\nu _p \left( \frac{q^{k+1} - 1}{q-1} \right)
   =\nu _p(q^{k+1}-1)
   =\nu _p(q^{rm}-1),
\end{equation}
and  if $\nu _p(\sigma(q^k))=s$ then $q^{r\cdot m}=p^s\ell '+1$ and $q^{r}=p^s\ell+1$, for 
some integers $\ell, \, \ell '$. Then 
$\begin{displaystyle}
    p^s=\sigma(q^k)
        =\frac{q^{k+1}-1}{q-1}
        =\frac{p^s\ell '}{q-1},
\end{displaystyle}$
and therefore $q=\ell '+1$. 

Then, for some $c \in \mathbb{N}$, it follows that 
$q^r-1=c\cdot \ell'$. Thus $p^s \ell  = q^r - 1 = c\cdot \ell '$ and  
$(\ell',p) = 1$ implies $\nu _p(c)=s$. Therefore 
$c/p^{s} \in \mathbb{Z}^{+}$ and therefore $\ell ' \leq \ell$. The inequality  $\ell \leq \ell '$ now implies 
$m=1$. The proof is complete.
\end{proof}

The first part of Theorem \ref{thm-types} is established next.  The hypothesis has four 
components and it is required to show $\nu_{p}(\sigma(q^{k})) \leq \lfloor \log_{p} q^{k} \rfloor$ in 
each case.

\begin{proof}
We consider each case individually,

\smallskip

\noindent
\texttt{Case 1}:   \texttt{Either} $p=q$ \texttt{or} $q\not \equiv 1 \pmod p$ \texttt{and} $r$ 
\texttt{does not divide} $k+1$. The result 
follows from Theorem \ref{thm-val2} since $0=\nu _p(\sigma (q^k))<\log _p(q^k).$

\smallskip

\noindent
\texttt{Case 2}.  \texttt{Assume $q \equiv 1\pmod{p}$}.  Then $q>p$ and Theorem \ref{thm-val2} yields 
\begin{equation}
\nu_p(\sigma(q^k))=\nu_p(k+1)  \leq k\cdot 1 
          <k\log_p(q) = \log_p(q^k).
\end{equation}

\smallskip

\noindent
\texttt{Case 3.  Suppose $k=1$ and $q>2$}. Define   $\alpha$ by
$p^\alpha < q < q+1 \leq p^{\alpha+1}$. If $q+1<p^{\alpha+1}$ then
$\begin{displaystyle}  \nu_p(\sigma(q))=\nu_p(q+1)\leq \alpha = \lfloor\log_p(n)\rfloor. 
\end{displaystyle}$
Also $q+1  \neq p^{\alpha + 1}$ since otherwise  $q=2, \, p=3$. 

\smallskip 

\noindent
\texttt{Case 4}. \texttt{$q\not \equiv 1\pmod p,$ $r$ divides $k+1$ and  $r \neq k+1.$}
Write $k=r\cdot \ell-1$ and observe that $r<k$. Now there are several cases to consider:
   \begin{itemize}
    \item $\nu_p(k+1)=0$. Then $\nu_p(\sigma(q^k))= \nu_p(q^r-1) \leq \log_p(q^r) < \log_p(q^k).$
    \item $\nu _p(k+1)\neq 0 \text{ and } q>p$. Then 
     \begin{eqnarray*}
    \nu _{p}(\sigma (q^k))&= & \nu _p  \left( \frac{k+1}{r} \right) +\nu _p(q^r-1)
    \leq   \left( \frac{k+1}{r}-1 \right) +\log _p(q^r)\\
      &=  & \log _p \left( q^{\frac{r\ell+r(r-1)}{r}} \right) 
    \leq   \log _p(q^k)  \quad  \text{ because }\ell +r\leq \ell r=k+1.
    \end{eqnarray*}

    \item $\nu_p(k+1)\not= 0$, and $q<p$ there are three possibilities :\\
       \begin{enumerate}
         \item $ \min\{\nu_p(k+1),\nu_p(q^r-1)\}\geq 2$:
             \begin{align*}
                \nu _p(\sigma (q^k))&=\nu _p(k+1)+\nu _p(q^r-1)\\
                &\leq \nu _p(k+1)\nu _p(q^r-1)  \,\,\,\,  \text{because }x+y\leq xy\text{ for }x,y\geq 2 & \\
                &=\nu _p(\ell)\nu _p(q^r-1)
                \leq (\ell -1)\nu _p(q^r-1) \leq (\ell -1)\log _p(q^r)
                 \leq \log _p(q^k).
                \end{align*}
            \item $\nu_p(k+1)= 1$:
              \begin{align*}
                \nu _p(\sigma (q^k))&=\nu _p(k+1)+\nu _p(q^r-1)\\
                &=1+\nu _p(q^r-1)\\
                &\leq \nu _p(q^{p-1}-1)+\nu _p(q^r-1)\\
                &\leq \log _p(q^k) & \text{because } r+p\leq rp\leq k+1.
            \end{align*}
            \item $\nu_p(k+1)\geq2, \nu_p(q^r-1)=1$:   one may  assume $\nu _p(k+1)\geq 2$ and hence 
            $k+1= p^2r\ell _1\geq 4\cdot 2\ell _1\geq 8$,  then
             \begin{align*}
             \nu _p(\sigma (q^k))&=\nu _p(k+1)+\nu _p(q^r-1)
              \leq \log _p(k+1)+1\\
                &\leq \log _p(2^k) \leq \log _p(q^k).
                \end{align*}
     \end{enumerate}
\end{itemize}
\noindent
This completes the proof. 
\end{proof}

The next step is to establish the second  part of Theorem \ref{thm-types}.  The assumptions are 
now $r= k+1$ and $q \not \equiv 1 \mod p$.  The claim is that 
$  \nu _p(\sigma (q^k))\leq \lceil \log _p(q^k)\rceil. $
\begin{proof}
The results of Theorem \ref{thm-val2}  are used to justify the steps in 
         \begin{align*}
             \nu_p(\sigma(q^k)) & =  \nu_p(q^r-1) 
            =\nu_p((q-1)(1+q+q^2+ \ldots + q^{r-1}))
            \leq  \log_p(1+q^2+\ldots + q^k)\\
            &=\log_p \left( q^k \left( \frac{1}{q^k} + \frac{1}{q^{k-1}}+ \ldots + 1 \right) \right) 
             <\log_p(q^k) + \frac{1}{\log{p}}\sum_{j=1}^{k}\frac{1}{q^j} \\
            & < \log_p(q^k) + \frac{1}{q-1} 
            < \log_p(q^k)+1.
         \end{align*}
    Note that $\log_p(q^k) \not\in \mathbb{Z}$ and this implies 
  $ \begin{displaystyle}
        \nu_p(\sigma(q^k)) \leq \lfloor \log_p(q^k) + 1 \rfloor = \lceil \log_p(q^k) \rceil.
    \end{displaystyle}$
\end{proof}

The main result of this section is described next.  It establishes the bound 
$\begin{displaystyle} \nu_{p}(\sigma(n)) \leq \lceil \log_{p}(n) \rceil \end{displaystyle}$ under the 
conditions stated in  Conjecture \ref{conj1}.  In terms of the statement of Theorem \ref{thm-main}, this 
conjecture states that every number $n \in \mathbb{N}$ satisfy one of the conditions $(1), \, (2), $ or 
$(3)$.

\begin{theorem}
\label{thm-main}
Let $p>2$ be a fixed prime number, $n \in \mathbb{N}$ and write 
$\alpha = \nu_{q}(n)$, so that $q^{\alpha}$ is 
exactly the part of $n$ containing the prime $q$. Assume that either
\begin{enumerate}
\item every prime $q$ dividing $n$ satisfies $\nu _p(\sigma(q^{\alpha})) \leq \lfloor \log _p(q^{\alpha})\rfloor$, or 
    \item every prime $q$ dividing $n$ which satisfies  $\nu _p(\sigma(q^{\alpha})) =\lceil \log _p(q^{\alpha})\rceil$ is 
    less than $p$, or
   \item there is a unique prime $q>p$ dividing $n$ such that  $\nu _p(\sigma(q^{\alpha}))=\lceil \log _p(q^{\alpha})\rceil$.
\end{enumerate}
\noindent
Then $\begin{displaystyle} \nu _p(\sigma (n))\leq \lceil \log _p(n)\rceil .\end{displaystyle}$
\end{theorem}
\begin{proof}
The prime factors of $n$ are partitioned into two disjoint groups:
\begin{eqnarray}
\mathcal{N}_{1} & = & \left\{ q \text{ prime } \, q | n \,\, \text{ and } \,\,  \nu_{p} \left( \sigma \left(q^{\nu_{q}(n)} \right) \right)  \leq 
\left\lfloor \log_{p} \left( q^{\nu_{q}(n)} \right) \right\rfloor \right\}  \\
\mathcal{N}_{2} & = &  \left\{ q \text{ prime } \, q | n \,\,  \text{ and } \,\,  \nu_{p} \left( \sigma \left( 
q^{\nu_{q}(n)}\right)  \right)  \geq 
\left\lceil \log_{p} \left( q^{\nu_{q}(n)} \right) \right\rceil \right\}. \nonumber
\end{eqnarray}

\smallskip 

The proof is divided in cases according to conditions on $\mathcal{N}_{2}$.   Case 1 assumes condition 
$(1)$. Case 2 proves that conditions $(1), (2)$ and $(3)$ imply the result. Finally, Case 3 proves the 
the result assuming  (2).

\smallskip

\noindent
\texttt{Case 1}. Assume $(1)$. Then $\mathcal{N}_{2}$ is empty and the conclusion 
follows directly from the inequality  $\lfloor x \rfloor + \lfloor y \rfloor\leq \lfloor x+y \rfloor$ and its 
generalization to several summands. 

\noindent
\texttt{Case 2}. Assume $(1), (2)$ and $(3)$. From $(3)$ there is a unique prime $q>p$ such that 
$\nu_{p}(\sigma(q^{\alpha})) = \lceil \log_{p}(q^{\alpha}) \rceil$. Define $N = n/q^{\alpha}$. Then 
$\gcd(N,q) = 1$ and the number $N$ satisfies conditions $(1)$ or $(2)$. Then $n = N q^{\alpha}$ and 
so $\sigma(n) = \sigma(N) \sigma(q^{\alpha})$ and 
\begin{eqnarray}
\nu_{p}(\sigma(n)) & = & \nu_{p}(\sigma(N)) + \nu_{p}(\sigma(q^{\alpha})) \\
& \leq & \lfloor  \log_{p}N \rfloor  + \lceil \log_{p} q^{\alpha} \rceil \nonumber \\
& \leq & \lceil \log_{p} N  + \log_{p} q^{\alpha} \rceil \nonumber  \quad \text{using } 
\lfloor x \rfloor + \lceil y  \rceil  \leq \lceil x+ y \rceil \nonumber \\
& = & \lceil \log_{p} n \rceil. \nonumber 
\end{eqnarray}

\smallskip

\noindent
\texttt{Case 3}. Finally assume condition $(2)$. Then split $n$ into its factorization according to these sets; that is, $n = n_{1} \cdot n_{2}$, with 
\begin{equation}
n_{1} = \prod_{q \in \mathcal{N}_{1}} q^{\nu_{q}(n)} \quad \text{and} \quad 
n_{2} = \prod_{q \in \mathcal{N}_{2}} q^{\nu_{q}(n)}.
\end{equation}

 \smallskip
 
 The rest of the proof is divided into a small number of steps.
 
 \smallskip 
 
 \noindent
 \texttt{Step 1}. $\nu _p(\sigma (n_1))\leq \lfloor \log _p(n_1)\rfloor$. This follows from 
 $\lfloor x \rfloor + \lfloor y \rfloor\leq \lfloor x+y \rfloor.$
 
 \medskip 
 
 \noindent
 \texttt{Step 2}. $\nu _p(\sigma(n_2)) \geq \lceil \log _p(n_2)\rceil$. 
 \begin{proof}
 Write $\begin{displaystyle} n_2=\prod _{i=1}^tq_i^{\alpha _i} \end{displaystyle}$, where $q_{i}$ are the 
 primes in $\mathcal{N}_{2}$ and $\alpha_{i} = \nu_{q_{i}}(n)$. Therefore 
 $ \nu_p(\sigma(q_i^{\alpha_i})) \geq \lceil\log_p(q_i^{\alpha_i})\rceil$. Lemma 
   \ref{lemmaequ} shows there are $s_i\in \mathbb{Z}^{+} $ such that $\sigma(q_i^{\alpha _i})=p^{s_i}.$
   Observe that
   \begin{equation}
   s_{i} = \log_{p}(\sigma(q_{i}^{\alpha_{i}}) )= \left \lceil  \log_{p}(\sigma(q_{i}^{\alpha_{i}}) )\right \rceil,
   \end{equation}
   \noindent
   the last equality is valid since $x = \lceil x \rceil$ for an integer $x$. 
   
   Therefore
      \begin{eqnarray}
   \nu_{p}(\sigma(n_{2})) & = & \nu_{p} \left( \prod_{i=1}^{t} \sigma(q_{i}^{\alpha_{i}}) \right) 
    =  \sum_{i=1}^{t} \nu_{p} ( \sigma(q_{i}^{\alpha_{i}})) \nonumber  \\
   & = & \sum_{i=1}^{t} s_{i} 
    =  \sum_{i=1}^{t} \left \lceil \log_{p}( \sigma(q_{i}^{\alpha_{i}}) ) \right \rceil 
    \geq  \sum_{i=1}^{t} \left \lceil \log_{p}(q_{i}^{\alpha_{i}}) \right \rceil \nonumber  \\
   & \geq &\left \lceil  \sum_{i=1}^{t} \log_{p}(q_{i}^{\alpha_{i}}) \right \rceil  \,\, \text{since} \,\, 
   \lceil x \rceil + \lceil y \rceil  \geq \lceil x+y \rceil
   \nonumber  \\
   & = & \left \lceil \log_{p} n_{2} \right \rceil. \nonumber
   \end{eqnarray}
   \end{proof} 
   
  \noindent
  \texttt{Step 3}.  $\nu _p(\sigma(n_2)) \leq \lceil \log _p(n_2)\rceil$
  \begin{proof}
  It is at this point that condition (2), on the prime divisors of $n$ is used. 
  
  As in Step 2, write $\begin{displaystyle} n_{2} = \prod_{i=1}^{t} q_{i}^{\alpha_{i}} \end{displaystyle}$ with 
  $ \nu_p(\sigma(q_i^{\alpha_i})) \geq \lceil\log_p(q_i^{\alpha_i})\rceil$.  Then 
      $$\nu _p(\sigma (n_2))=\sum _{i=1}^t \left( \log _p(q_i^{\alpha _i})+
    \log_p \left( \frac{\sigma (q_i^{\alpha _i} )}{q_i^{\alpha _i}} \right) \right)
    =\log _p n_2 +\sum _{i=1}^t\log _p\left (1+\sum _{j=1}^{\alpha _i}\frac{1}{q_i^j}\right ),$$
    and $\log (1+x)<x$ yields 
    \begin{equation}
    \label{bound1}
    \nu _p(\sigma (n_2))< \log _p n_2  + \frac{1}{\log p}\sum _{i=1}^t\left (\sum _{j=1}^{\alpha _i}\frac{1}{q_i^{j}}\right )
    < \log _p n_2 + \frac{1}{\log p}\sum _{i=1}^t\frac{1}{q_i-1}.
    \end{equation}
  To bound the last sum, observe that for $p\geq 11$,
  \begin{equation}
  H_{10} = \sum_{q=2}^{11} \frac{1}{q-1} 
  = \sum_{\substack{q = 2  \\ q \, \text{prime}}}^{11}  \frac{1}{q-1}  + 
  \sum_{\substack{q = 2  \\ q \, \text{not prime}}}^{11}  \frac{1}{q-1} 
  \end{equation}
  \noindent
  The sum on the right is $\tfrac{1}{3} + \tfrac{1}{5} + \tfrac{1}{7} + \tfrac{1}{8} + \tfrac{1}{10} = \tfrac{757}{840} > 0.9$. Then 
  \begin{eqnarray}
  H_{p-1} & = &  \sum_{k=1}^{p-1} \frac{1}{k}  \\
 & = &  \sum_{\substack{k = 2 \\ k \text{ not prime}}}^{p}   \frac{1}{k-1} + 
   \sum_{\substack{k = 2 \\ k \text{ prime}}}^{p}  \frac{1}{k-1}  \nonumber  \\
   & > &  0.9 + \sum_{i=1}^{t} \frac{1}{q_{i}-1}. \nonumber 
   \end{eqnarray}
   \noindent
   The first bound follows from $p \geq 11$ and it was shown above that the corresponding sum up to 
   $11$ is bounded by $0.9$. The second bound comes from the assumption $(2)$, that all the 
   primes $q_{i}$, 
   appearing in the factor $n_{2}$, satisfy $q_{i} < p$. 
   
   Now recall the value 
   \begin{equation}
   H_{m} = \psi(m) + \gamma
   \end{equation}
   \noindent
   where $\psi(x) = \Gamma'(x)/\Gamma(x)$ is the digamma function and $\gamma$ is Euler's constant. 
   (See \cite{andrews-1999a} for details). The 
   standard bound $\psi (m) \leq \log m -\frac{1}{2m}$ yields
 $$H_{p-1}\leq \log p +\gamma -\frac{1}{2p}<\log p+0.9, \,\, \text{since } p \geq 11 \,\, \text{and } 
 \gamma  \sim 0.57721.$$
Therefore 
    $$\sum _{i=1}^t\frac{1}{q_i-1}<H_{p-1}-0.9<\log p, \,\,\, \text{for } p \geq 11,$$
 and then 
\begin{align*}
  \nu _p(\sigma (n_2))< \log _p n_2 + \frac{1}{\log p}(H_{p-1}-0.9)
    <\log_p n_{2}+1.
\end{align*}
This establishes the bound 
$\nu_{p}(\sigma(n_{2})) \leq \lceil \log_{p}(n_{2}) \rceil$  for $p \geq 11$. The missing cases 
$p=3, \, 5, \, 7$  can be checked directly. This finishes  the discussion of Case 3.
\end{proof}

\begin{note}
Step 2 and Step 3 show that, under the assumptions stated in Conjecture \ref{conj1}, the 
identity 
\begin{equation}
 \nu _p(\sigma(n_2))  = \lceil \log _p(n_2)\rceil
 \end{equation}
 \noindent
 holds. 
\end{note}

\smallskip 
 
 \noindent
 \texttt{Step 4} completes the proof of Theorem \ref{thm-main}. Write $n = n_{1} \cdot n_{2}$, with 
 the above notation. Then from $\nu_{p}(\sigma(n_{1})) \leq \lfloor \log_{p} n_{1} \rfloor$ and 
 $\nu_{p}(\sigma(n_{2})) = \lceil \log_{p} n_{2} \rceil$, it follows that
 \begin{eqnarray}
 \nu_{p}(\sigma(n)) &= & \nu_{p}(\sigma(n_{1}n_{2})) \\
  & = & \nu_{p}( \sigma(n_{1}) \sigma(n_{2})) \nonumber \\
  & = & \nu_{p}(\sigma(n_{1})) + \nu_{p}(\sigma(n_{2}))  \nonumber  \\
  & \leq & \lfloor \log_{p}n_{1} \rfloor + \lceil \log_{p} n_{2} \rceil \nonumber \\ 
  & \leq & \lceil \log_{p}n_{1} + \log_{p}n_{2} \rceil   
  \,\, \text{using} \,  \lfloor x \rfloor  + \lceil y \rceil \leq \lceil x+y \rceil \nonumber \\
  & = & \lceil \log_{p} n \rceil. \nonumber 
  \end{eqnarray}
   The proof is complete.
  \end{proof}
  
  \newpage 
  
  Theorem \ref{thm-main} gives the bound $\nu_{p}(\sigma(n)) \leq \lceil \log_{p} n \rceil$. The question of 
  when is the equality achieved is discussed next. 
  
  \begin{lemma}
  \label{lemma-bd1}
  Let $n = q_{1}^{\alpha_{1}} \cdots q_{\ell}^{\alpha_{\ell}}$ be the prime factorization of $n$. Assume that every 
  prime component $q_{j}^{\alpha_{j}}$ satisfies the equality 
  \begin{equation}
  \label{equ-pow1}
  \nu_{p}( \sigma(q_{j}^{\alpha_{j}} )) =  \lceil \log_{p} q_{j}^{\alpha_{j}} \rceil.
  \end{equation}
  \noindent
  Then $n$ achieves the upper bound $\nu_{p}(\sigma(n)) = \lceil \log_{p} n \rceil$ in Theorem \ref{thm-main}.
  \end{lemma}
  \begin{proof}
  The result comes from 
  \begin{equation*}
  \lceil \log_{p} n \rceil \geq \nu_{p}(\sigma(n)) =  \sum_{j=1}^{\ell} \nu_{p}( \sigma(q_{j}^{\alpha_{j}}))  = 
  \sum_{j=1}^{\ell} \lceil \log_{p} q_{j}^{\alpha_{j}} \rceil  \geq \left \lceil \sum_{j=1}^{\ell} \log_{p} q_{j}^{\alpha_{j}} 
  \right \rceil = \lceil \log_{p} n \rceil. 
  \end{equation*}
  \end{proof}
  
  The converse of Lemma \ref{lemma-bd1} is not valid. For example, take $n = 24400 = 2^{4} \cdot 5^{2} 
  \cdot 61$. Then $\nu_{31}(\sigma(n)) = 3 = \lceil \log_{31}n \rceil$ and the factor $61$ satisfies 
  $\nu_{31}(\sigma(61)) = 1$ and $\lceil \log_{31} 61 \rceil  = 2$. 
  
  \medskip
  
  \begin{note}
  The result above gives a procedure to find some integers $n$ where the maximum bound is achieved. Theorem \ref{thm-main} shows that in the
   case of prime powers,  equality  \eqref{equ-pow1}   is equivalent to condition (2) in 
  Lemma \ref{lemmaequ}. Therefore
\begin{equation}
\label{dio1}
\frac{q^{k}-1}{q-1} = p^{s}.
\end{equation}
\noindent
with $q = q_{j}$ and $k = \alpha_{j}+1$.  Thus $p, \, q$ are prime solutions to the so-called Ljunggren-Nagell  diophantine equation 
\begin{equation}
\frac{x^{k}-1}{x-1} = y^{s}.
\label{nl-equation}
\end{equation}
\noindent
Examples include 
\begin{equation*}
\frac{2^3-1}{2-1} = 7, \quad \frac{5^3-1}{5-1} = 31, \quad \frac{3^5-1}{3-1} = 11^2, \quad
\frac{7^4-1}{7-1} = 20^{2}, \quad \frac{18^{3}-1}{18-1} = 7^3.
\end{equation*}
\noindent
According to \cite{bugeaud-2007a}, `presumably the last three examples are the only solutions of 
\eqref{nl-equation} with $r>1$ and $s \geq 2$'. Information about \eqref{dio1} appears in  
\cite{bennettm-2015a} and \cite{bugeaud-2002a}.
\end{note}

\section{A property of the prime $3$}
\label{sec-prop3}

The goal of this section is to characterize the indices $n$ for which $\nu_{3}(\sigma(n)) = 0$.

\begin{theorem}
\label{thm-theta1}
For $n \in \mathbb{N}$, the valuation $\nu_{3}(\sigma(n)) = 0$ if and only if $n$ is represented by the 
quadratic form $Q(b,c) = b^{2}+bc+c^{2}$; that is, there are $b, \, c \in \mathbb{Z}$ such that 
$Q(b,c) = n$.
\end{theorem}

The proof  is divided into a number of steps.

\begin{proposition}
\label{sec3-prop1}
Let $\begin{displaystyle} \theta(q) = \prod_{k=1}^{\infty} (1 - q^{k}) \end{displaystyle}$ and 
$\begin{displaystyle} G(q) = \sum_{b,c \in \mathbb{Z}} \omega^{b-c} q^{b^{2}+bc+c^{2}} 
\end{displaystyle}$ where $\omega = e^{2 \pi i/3}$. Then 
\begin{equation}
\frac{\theta^{3}(q)}{\theta(q^{3})} = G(q). 
\label{Gdef}
\end{equation}
\end{proposition}
\begin{proof}
Introduce the standard notation 
\begin{equation}
(a;q)_{n} = \prod_{k=1}^{n} (1 - aq^{k-1}) \quad \text{and} \quad 
(a;q)_{\infty} = \prod_{k=1}^{\infty} (1 - aq^{k-1}),
\end{equation}
\noindent
and recall a  classical result of Euler (see \cite{andrews-1976a})
\begin{equation}
\label{prop3-eq1}
(-x;q)_{\infty} = \prod_{n=0}^{\infty} (1 + xq^{n}) = \sum_{k=0}^{\infty} \frac{q^{\binom{k}{2}}}{(q;q)_{k}}
x^{k}.
\end{equation}
\noindent
The identity $(1+t^3) = (1+t)(1+ \omega t)(1+\omega^{2}t)$ now leads to 
\begin{equation}
(-x^3;q^3)_{\infty} = (-x;q)_{\infty} (- x \omega;q)_{\infty} (-x \omega^{2};q)_{\infty}
\end{equation}
\noindent
and \eqref{prop3-eq1} gives 
\begin{equation}
\sum_{n=0}^{\infty} \frac{q^{3 \binom{n}{2}}}{(q^{3};q^{3})_{n}} x^{3n} = 
\sum_{\ell,j,k \geq 0} \omega^{j + 2k } 
\frac{q^{ \binom{\ell}{2} + \binom{j}{2} + \binom{k}{2}}}{(q;q)_{\ell} (q;q)_{j} (q;q)_{k}} x^{\ell + j +k }.
\end{equation}
\noindent
Matching powers of $x$ leads to 
\begin{equation}
\label{prop3-eqn2}
\frac{1}{(q^{3};q^{3})_{n}} = \sum_{\ell+j+k= 3n} \omega^{j-2k} 
\frac{q^{  \binom{\ell}{2} + \binom{j}{2} + \binom{k}{2} - 3 \binom{n}{2}}}
{ (q;q)_{\ell} (q;q)_{j} (q;q)_{k}}.
\end{equation}
\noindent
The substitution $a = \ell-n, \, b = j-n, \, c = k-n$ transforms \eqref{prop3-eqn2} into 
\begin{equation}
\frac{1}{(q^{3};q^{3})_{n}} = \sum_{a+b+c=0} \omega^{b-c} 
\frac{q^{\tfrac{1}{2} (a^{2}+b^{2}+c^{2})}}
{(q;q)_{a+n} (q;q)_{b+n} (q;q)_{c+n}}.
\end{equation}
\noindent
Now let $n \rightarrow \infty$ to produce 
\begin{equation}
\frac{1}{(q^{3};q^{3})_{\infty}} = \sum_{a+b+c=0} \omega^{b-c} 
\frac{q^{\tfrac{1}{2} (a^{2}+b^{2}+c^{2})}}{(q;q)_{\infty}^{3}}
\end{equation}
\noindent
and hence 
\begin{equation}
\frac{\theta^{3}(q)}{\theta(q^{3})} = \frac{(q;q)_{\infty}^{3}}{(q^{3};q^{3})_{\infty}} = 
\sum_{b, c \in \mathbb{Z}} \omega^{b-c} q^{b^{2}+bc+c^{2}} = G(q).
\end{equation}
\end{proof}

Define $\begin{displaystyle} F(q) = \sum_{b, c \,  \in \mathbb{Z}} q^{b^{2}+bc+c^{2}} \end{displaystyle}$
and introduce the coefficients $\beta_{n}$ by the expansion 
$\begin{displaystyle} F(q) = \sum_{n \geq 0} \beta_{n} q^{n}. \end{displaystyle}$  The 
corresponding coefficients 
$\alpha_{n}$ are defined for the function $G(q)$ in Proposition \ref{sec3-prop1} by 
$\begin{displaystyle} G(q) = \sum_{n \geq 0} \alpha_{n} q^{n}. \end{displaystyle}$  The next 
statement connects these two sequences. 

\begin{corollary}
The relation 
\begin{equation}
\beta_{n} = \begin{cases}
\,\,\,\,\,\,\, \alpha_{n} & \quad \text{if} \, \, \beta_{n} \geq 0, \\
-2 \alpha_{n} & \quad \text{if} \,\, \beta_{n} < 0. 
\end{cases}
\end{equation}
\noindent
If particular, $\beta_{n} = 0$ is equivalent to $\alpha_{n} = 0$ and $\beta_{n} \neq 0$ if and 
only if $n = b^{2}+bc+c^{2}$ for some $b, \, c \in \mathbb{Z}$.
\end{corollary}
\begin{proof}
Introduce the sets 
$$ \mathcal{A} = \{ (b,c) \in \mathbb{Z}^{2}: \, b \equiv c  \bmod 3\} \,\, \text{and} \,\, 
 \mathcal{A}' = \{ (b,c) \in \mathbb{Z}^{2}: \, b \not  \equiv c  \bmod 3\}.
$$
\noindent
Then $\mathcal{A} \cup \mathcal{A}' = \mathbb{Z}^{2}$ and 
$\mathcal{A} \cap \mathcal{A}' = \{ \,\, \}$.  A direct calculation shows that the congruence 
$b^{2}+bc+c^{2} \equiv 0 \bmod 3$ is equivalent to $(b,c) \in \mathcal{A}$. Now using 
$\realpart{\omega} = \realpart{\omega^{2}} = - \tfrac{1}{2}$, it follows that 
\begin{eqnarray}
F(q) & = & \sum_{n \geq 0} \left( \sum_{\substack{(b,c) \in \mathcal{A}\\ \,\, b^{2}+bc+c^{2} = n}}1  \right)
q^{n} + 
\sum_{n \geq 0} \left( \sum_{\substack{(b,c) \in \mathcal{A}'\\ \,\, b^{2}+bc+c^{2} = n}}1  \right)
q^{n}  \\
& = & \sum_{\substack{\,\, b^{2}+bc+c^{2} = n \\ n \equiv 0 \bmod 3}} \beta_{n}q^{n} + 
 \sum_{\substack{\,\, b^{2}+bc+c^{2} = n \\ n \not \equiv 0 \bmod 3}} \beta_{n}q^{n}
 \nonumber
 \end{eqnarray}
 \noindent
 and 
 \begin{eqnarray}
G(q) & = & \sum_{n \geq 0} \left( \sum_{\substack{(b,c) \in \mathcal{A}\\ \,\, b^{2}+bc+c^{2} = n}}1  \right)
q^{n} - \frac{1}{2}  
\sum_{n \geq 0} \left( \sum_{\substack{(b,c) \in \mathcal{A}'\\ \,\, b^{2}+bc+c^{2} = n}}1  \right)
q^{n}  \\
& = & \sum_{\substack{\,\, b^{2}+bc+c^{2} = n \\ n \equiv 0 \bmod 3}} \beta_{n}q^{n}  - \frac{1}{2}
 \sum_{\substack{\,\, b^{2}+bc+c^{2} = n \\ n \not \equiv 0 \bmod 3}} \beta_{n}q^{n}.
 \nonumber
 \end{eqnarray}
 \noindent
 The result follows directly from here.
\end{proof}

The next statement gives a Lambert series expression for the function $F$.  The identity given below
is classical: it appears in a famous letter from Ramanujan to Hardy in $1917$. An alternative proof
appears in Chapter 21 of \cite{hirschhorn-2017a}.

\begin{corollary}
\label{coro-3}
The identity 
\begin{equation}
F(q) = 1 + 6 \sum_{m=0}^{\infty} \left( \frac{q^{3m+1}}{1-q^{3m+1}} -  \frac{q^{3m+2}}{1-q^{3m+2}} 
\right)
\end{equation}
\noindent
holds.
\end{corollary}
\begin{proof}
After some change of variables, the proof of Theorem \ref{thm-theta1} gives a new expression for 
$G(q) = \theta^{3}(q)/\theta(q^{3})$ in the  form 
\begin{eqnarray*}
G(q) & = & \sum_{k, \ell \in \mathbb{Z}} q^{3k^{2} + 3 k \ell + 3 \ell^{2}} + 
\omega \sum_{k, \ell \in \mathbb{Z}} q^{3k^{2} + 3 k \ell + 3 \ell^{2} + 3k + 3 \ell + 1} + 
\overline{\omega} \sum_{k, \ell \in \mathbb{Z}} q^{3k^{2} + 3 k \ell + 3 \ell^{2} + 3k + 3 \ell + 1}  \\
& = & F(q^{3}) - q \sum_{k, \ell \in \mathbb{Z}} q^{3k^{2} + 3 k \ell + 3 \ell^{2} + 3k + 3 \ell}. 
\nonumber 
\end{eqnarray*}
\noindent
Now invoke Jacobi's identity 
\begin{equation}
\theta^{3}(q) = \sum_{n \geq 0} (-1)^{n} (2n+1) q^{\binom{n+1}{2}}
\end{equation}
\noindent
and splitting up this sum according to the residue of $n$ modulo $3$ yields 
\begin{eqnarray}
\theta^{3}(q) & = & 
\sum_{n \geq 0} (-1)^{n} (6n+1) q^{\tfrac{9n^{2}+3n}{2}} - 
\sum_{n \geq 0} (-1)^{n} (6n+3) q^{\tfrac{9n^{2}+9n+2}{2}}  \\
& & - 
\sum_{n < 0} (-1)^{n} (-6n-1) q^{\tfrac{9n^{2}+3n}{2}}  \nonumber \\
& = & \sum_{n \in \mathbb{Z}} (-1)^{n} (6n+1) q^{\tfrac{9n^{2}+3n}{2}} - 
3q \theta(q^{9})^{3}, \nonumber 
\end{eqnarray}
\noindent
and hence, from \eqref{Gdef}, 
\begin{eqnarray}
G(q) & = &  F(q^{3}) - q \sum_{k, \ell \in \mathbb{Z}} q^{3k^{2} + 3 k \ell + 3 \ell^{2} + 3 k + 3 \ell} 
\label{GFiden} \\
& = & \frac{1}{\theta(q^{3})} 
\left( \sum_{n \in \mathbb{Z}} (-1)^{n} (6n+1) q^{\tfrac{9n^{2}+3n}{2}} - 3q \theta(q^{9})^{3} \right). 
\nonumber 
\end{eqnarray}
\noindent
Now replace $q$ by $q^{1/3}$ and match the series with integer exponents to obtain 

\begin{equation}
F(q) = \frac{1}{\theta(q)} \sum_{n \in \mathbb{Z}} (-1)^{n} (6n+1)q^{\tfrac{3n^{2}+n}{2}}.
\end{equation}

\bigskip 

Consequently, it follows that 
\begin{eqnarray*}
F(q) & = & \frac{1}{\theta(q)} \sum_{n \in \mathbb{Z}} (-1)^{n} (6n+1) q^{\tfrac{3n^{2}+n}{2}} \\
& = & \frac{1}{\theta(q)} \left[ \frac{d}{dt} \sum_{n \in \mathbb{Z}} (-1)^{n} t^{6n+1} q^{\tfrac{3n^{2}+n}{2}} 
\right]_{t=1} \nonumber \\
&= & \frac{1}{\theta(q)} \left[ \frac{d}{dt} \left\{ t \prod_{m=1}^{\infty} 
(1 - t^{6} q^{3m-1}) (1 - t^{-6} q^{3m-2}) (1 - t^{6} q^{3m}) \right\} \right]_{t=1} \nonumber \\
& = & \frac{1}{\theta(q)}
\left[ \prod_{m=1}^{\infty} (1 - t^{6} q^{3m-1} ) (1 - t^{-6}q^{3m-2})(1-q^{3m}) \right]_{t=1} \nonumber \\
& & \quad \quad \quad  \times \left\{ 1 + \sum_{m=1}^{\infty} \frac{t^{-6} q^{3m-2}}{1- t^{-6} q^{3m-2}} - 
\frac{t^{6} q^{3m-1}}{1 - t^{6} q^{3m-1}} \right\}_{t=1}  \nonumber \\
& = & 1 + 6 \sum_{m=1}^{\infty} \left( \frac{q^{3m-2}}{1-q^{3m-2}} - \frac{q^{3m-1}}{1-q^{3m-1}} \right), 
\nonumber 
\end{eqnarray*}
\noindent
using Jacobi's triple identity to cancel $\theta(q)$ with the infinite product.
\end{proof}

The proof of Theorem \ref{thm-theta1} is now completed: start by writing the expression for 
$F(q)$ in Corollary \ref{coro-3} in terms of the Jacobi symbol 
$\left(\frac{x}y\right)$ as
\begin{equation}
\sum_{b,c\in\mathbb{Z}}q^{b^2+bc+c^2}=1+6\sum_{n\geq1}q^n\sum_{d\vert n}\left(\frac{d}3\right).
\end{equation}
On the other hand, 
\begin{eqnarray}
\sigma(n) & = & \sum_{d \vert n}  d  \\
& \equiv & \sum_{\substack{{d \vert n} \\ {d \equiv 1 \bmod 3}} }  1 + 
 \sum_{\substack{{d \vert n} \\ {d \equiv 2 \bmod 3}} }  2  \nonumber \\
 & \equiv & \sum_{\substack{{d \vert n} \\ {d \equiv 1 \bmod 3}} }  1  - 
 \sum_{\substack{{d \vert n} \\ {d \equiv 2 \bmod 3}} }  2  \nonumber \\
 & = & \sum_{d \vert n} \left(\frac{d}3\right), \nonumber
\end{eqnarray}
\noindent
instantly implies the assertion. 

\begin{note}
An expression for the prime factorization of the integers $n \in \mathbb{N}$ such that 
$\nu_{3}(\sigma(n)) = 0$ is given next.  Since $p=3$, the value of $r= \text{Ord}_{3}(q)$ is 
\begin{equation}
r = \begin{cases}
1 & \quad \text{if} \quad q \equiv 1 \bmod 3 \\
2 & \quad \text{if} \quad q \not \equiv 1 \bmod 3
\end{cases}
\end{equation}

Recall that Theorem \ref{thm-val2} gives 
\begin{equation*}
    \nu _p\left (\sigma (q^k)\right ) =\begin{cases}
    \nu _p (k+1), & \text{if }q\equiv 1 \pmod p\\
    0, & \text{if} \,\, p=q \text{ or }(q\not \equiv 1 \bmod p \text{ and }k\not \equiv -1 \bmod r)\\
\nu _p\left (k+1\right )+\nu _p(q^r-1), & \text{otherwise},
    \end{cases}
\end{equation*}
\noindent
motivating the writing of the prime factorization of $n$ in the form 
\begin{equation}
n = 2^{\alpha_{2}} 3^{\alpha_{3}} \prod_{\substack{j=1 \\ p_{j} \equiv 1 \bmod 3}}^{J_{1}} p_{j}^{\beta_{j}} 
\prod_{\substack{j=1 \\ q_{j} \not \equiv 1 \bmod 3  \\  \gamma_{j} \not \equiv -1 \mod r  \\ q_{j} \neq 3}}^{J_{2}} 
q_{j}^{\gamma_{j}}
\prod_{\substack{j=1 \\ r_{j} \not \equiv 1 \bmod 3  \\  \delta_{j}  \equiv -1 \mod r  \\ r_{j} \neq 3}}^{J_{3}} 
r_{j}^{\delta_{j}}.
\end{equation}
\noindent
Then 
\begin{equation}
\sigma(2^{\alpha_{2}}) = 2^{\alpha_{2} +1} - 1
\end{equation}
\noindent
and so this factor contributes to $\nu_{3}(\sigma(n))$ if $\alpha_{2}$ is odd. Therefore, if $\nu_{3}(\sigma(n)) 
= 0$, it follows that $\alpha_{2}$ is even.  Theorem \ref{thm-val2} shows that there is no restriction on 
$\alpha_{3}$.  The primes $p_{j}$ are congruent to $1 \bmod 3$, hence $p_{j} \neq 3$.  Of course the 
primes $p_{j}$ satisfy $p_{j} \equiv 1 \bmod 6$. Therefore 
\begin{equation}
\nu_{3}(\sigma( p_{j}^{\beta_{j}} )) =  \nu_{3}(\beta_{j}+1).
\end{equation}
\noindent
It follows that $\nu_{3}(\sigma(n)) = 0$ requires $\beta_{j} \equiv 0, \, 1 \bmod 3$.  The primes $q_{j}$ 
satisfy $q_{j} \equiv 2 \bmod 3$ and then $\gamma_{j}$ must be even, since $r = 2$. The same is true 
for the primes $r_{j}$.  Therefore $\nu_{3}(\sigma(n)) = 0$ precisely when 
$$n=3^a\prod _{i}\left (6k_i+1\right )^{\alpha _i}\cdot \prod _{j}\left (3\ell _j+2\right )^{2\beta _j},$$
where $a \in \mathbb{N}, \, \alpha _i\not \equiv 2 \pmod 3$ and $6k_i+1$ and $3\ell _j+2$ are prime numbers. For example, take $n = 10003  = 7 \cdot 1429$. Then 
$\sigma(n) = 11440 = 2^4 \cdot 5 \cdot 11 \cdot 13$ and 
$\nu_{3}(\sigma(n)) = 0$.
\end{note}

\begin{note}
The corresponding result for odd primes is the subject of current investigations. It is simple to produce 
parametrized families of numbers $n$ such that $\nu_{p}(\sigma(n) )\neq 0$. For example, consider the 
prime $p=5$ and  the family of numbers  
$\{ n = 8(2m+1): \, m \in \mathbb{N} \}$. Since $\sigma(8) = 15$ and $\gcd(8,2m+1) = 1$, it follows that 
\begin{equation}
\nu_{5}(\sigma(8(2m+1))) = \nu_{5}(\sigma(8)) + \nu_{5}(\sigma(2m+1)) \geq 1. 
\end{equation}
\noindent
Similarly, $\sigma(19) = 20$ implies 
\begin{equation*}
\nu_{5}(\sigma(19(19m+j))) = 1 + \nu_{5}(19m+j) \neq 0, \,\, \text{for fixed }j \text{ in the range }1 \leq j \leq 18.
\end{equation*}
\noindent
The characterization of integers $r$ such that $\sigma(r) \equiv 0 \bmod 5$ is 
an interesting question.  This sequence starts with 
$$\{ 8, \, 19, \, 24, \, 27, \, 29, \, 38, \, 40, \, 54, \, 56, \, 57, \, 58, \, 59, \, 72, \, 76, \, 79, \, 87, \, 88, \, 89, \, 95 \}.$$
\end{note}

\section{Conclusions}
\label{sec-concluions}

For a prime $p$ and $a \in \mathbb{N}$, the $p$-adic valuation $\nu_{p}(a)$ is the highest power of $p$ 
dividing $a$. For the function $\sigma(n)$, the sum of the divisors of $n$, expressions for $\nu_{p}(\sigma(n))$
are provided. These lead to the bound $\nu_{2}(\sigma(n)) \leq \lceil \log_{2} n \rceil$, with equality 
precisely when $n$ is the product of distinct Mersenne primes. For $p$ odd, the bound 
$\nu_{p}(\sigma(n)) \leq \lceil \log_{p} n \rceil$ is established  and equality holds when the prime factors 
of $n$ produce solutions to the Ljunggren-Nagell diophantine equation.  This last result requires an extra 
(technical) condition on $n$. It is conjectured that this condition is always valid.

\medskip

\noindent
\texttt{Acknowledgments}.  The authors wish to thank Christophe Vignat for suggestions to an earlier version 
of the manuscript.

%
%

\end{document}